\documentclass{amsart}
\usepackage[all,cmtip,ps]{xy}
\xyoption{rotate}
\usepackage{amsmath}
\usepackage{amsthm}
\usepackage{amssymb}
\usepackage{amscd}
\usepackage{hyperref}
\usepackage{dashrule}
\newcommand{\cA}{\mathcal{A}}

\newcommand{\cC}{\mathcal{C}}

\newcommand{\cD}{\mathcal{D}}

\newcommand{\cG}{\mathcal{G}}

\newcommand{\cT}{\mathcal{T}}

\newcommand{\cH}{\mathcal{H}}

\newcommand{\cU}{\mathcal{U}}

\newcommand{\cS}{\mathcal{S}}

\newcommand{\bZ}{\mathbb{Z}}

\DeclareMathOperator{\hocolim}{Hocolim}

\DeclareMathOperator{\Hom}{Hom}
\DeclareMathOperator{\End}{End}
\DeclareMathOperator{\Ext}{Ext}
\DeclareMathOperator{\Tor}{Tor}
\DeclareMathOperator{\Ker}{Ker}

\DeclareMathOperator{\Susp}{Susp}

\newcommand*{\rMod}{\textrm{\textup{Mod-}}}

\newcommand*{\rmod}{\textrm{\textup{mod-}}}
\newtheorem{theorem}{Theorem}[section]
\newtheorem{lemma}[theorem]{Lemma}
\newtheorem{proposition}[theorem]{Proposition}
\newtheorem{corollary}[theorem]{Corollary}
\theoremstyle{definition}

\newtheorem{remark}[theorem]{Remark}

\def\dualita#1#2{\mathrel{
                 \mathop{\vcenter{
                 \offinterlineskip
                 \hbox to 1.2truecm{\rightarrowfill}
                 \hbox to 1.2truecm{\leftarrowfill}}}
                 \limits_{#2}^{#1}}}
\def\dual{\mathrel{
                 \mathop{\vcenter{
                 \offinterlineskip
                 \hbox to .6truecm{\rightarrowfill}
                 \hbox to .6truecm{\leftarrowfill}}}
                 }}
                 
\def\cDll#1{{\cD}^{\leq #1}}
\def\cDgg#1{{\cD}^{\geq #1}}

\begin{document}
\title[On a property of $t$-structures generated by tilting modules]{On a property of $t$-structures generated by non-classical tilting modules}
\author[F. Mattiello]{Francesco Mattiello}

\address{ Dipartimento di Matematica, Universit\`a degli studi di Padova, via Trieste 63, I-35121 Padova, Italy}
\email{mattiell@math.unipd.it}
\date{}
\thanks{The author is supported by Assegno di ricerca ``Tilting theory in triagulated categories'' del Dipartimento di Matematica dellÕUniversit\`a degli Studi di Padova and by Progetto di Eccellenza della fondazione Cariparo.\\
The author thanks his institution for its help. He also thanks Manuel Saor\'in for the discussions on the topic.}

\begin{abstract}
Let $R$ be a ring and $T \in \rMod R$ be a (non-classical) tilting module of finite projective dimension. Let 
$\cT=(\cT^{\leq0}, \cT^{\geq0})$ be the $t$-structure on $D(R)$ generated by $T$ and $\cD=(\cD^{\leq0}, \cD^{\geq0})$ be the natural $t$-structure. We show that the pair $(\cD, \cT)$ is right filterable in the sense of \cite{FMT}, that is, for any $i\in\mathbb Z$ the intersection $\cD^{\geq i}\cap \cT^{\geq 0}$ is the co-aisle of a $t$-structure. As a consequence, the heart of $\cT$ is derived equivalent to $\rMod R$.

\end{abstract}

\maketitle

\section*{Introduction}

Tilting theory has its roots in the representation theory of finite dimensional algebras and has been widely used for constructing equivalences between categories. It is now considered an essential tool of many areas of Mathematics, including group theory, commutative and non-commutative algebraic geometry, and algebraic topology. Tilting theory appeared in the early seventies, when Bernstein, Gelfand and Ponomarev investigated the reflection functors in connection with giving a new proof of Gabriel's theorem~\cite{Gab}. This work was later generalized by Brenner and Butler in~\cite{zbMATH03697330}, who introduced the actual notion of a tilting module for finite dimensional and Artin algebras $A$ and the resulting tilting theorem between $\rmod A$ and the finitely generated modules over the endomorphism ring of a tilting 
$A$-module, a generalization of the so-called Morita Theory. Later the work of Brenner and Butler was simplified and generalized by Happel and Ringel~\cite{HR82}, Miyashita~\cite{Miya} and Colby and Fuller~\cite{CF90} to the case of modules over arbitrary rings. Given a tilting $A$-module $T$ of projective dimension $1$, there is a pair of equivalences $\Ker(\Ext^1_A(T, -)) \stackrel{\sim}\to \Ker(\Tor_B^1 (-, T ))$ and 
$\Ker(\Hom(T, -)) \stackrel{\sim}\to \Ker(- \otimes_BT)$ between the members of the torsion pairs $(\Ker(\Ext^1_A(T, -)),\Ker(\Hom_A(T, -)) )$ of $A$-modules and 
$(\Ker(- \otimes_BT),\Ker(\Tor_B^1 (-, T )))$ of $B$-modules, where $B=\End_A(T)$.

Soon it became clear that tilting modules and, more generally, tilting objects had an important role in the framework of derived categories and $t$-structures. Happel, Reiten and Smal{\o}~\cite{MR1327209} introduced a technique to construct, starting from a given $t$-structure $\cD$ and a torsion pair on its heart, a new $t$-structure, called the {\it tilted} $t$-structure with respect to the given torsion pair. In particular, given a classical $1$-tilting object $T$ in a Grothendieck category $\cG$, the heart of the $t$-structure obtained by tilting the natural $t$-structure in the derived category $D(\cG)$ with respect to the torsion pair generated by $T$ is derived equivalent to $\cG$. 

In~\cite{FMT} the authors generalise the Happel-Reiten-Smal{\o} result and recover the torsion torsion-free decomposition, passing from classical $1$-tilting objects to classical $n$-tilting objects. Given a Grothendieck category $\cG$ with a classical $n$-tilting object $T$, denoted by $\cT=(\cT^{\leq0},\cT^{\geq0})$ the $t$-structure on 
$D(\cG)$ generated by $T$ and by $\cD=(\cD^{\leq0},\cD^{\geq0})$ the natural $t$-structure, the pair $(\cD, \cT)$ is {\it (right) $n$-tilting}, that is: 1) $\cD^{\leq {-n}} \subseteq \cT^{\leq {0}} \subseteq \cD^{\leq {0}}$, and this relation is ``strict''; 2) it is {\it (right) filterable}, that is, for any $i\in\mathbb Z$ the intersection $\cD^{\geq i}\cap \cT^{\geq 0}$ is a co-aisle of a $t$-structure; 3) if $\cH_\cT$ denotes the heart of $\cT$, then $\cG\cap \cH_\cT$ is a cogenerating class of $\cG$, that is, each object of $\cG$ embeds in an object of $\cG\cap \cH_\cT$. An important consequence of this fact is that the natural inclusion $\cH_\cT \hookrightarrow D(\cG)$ extends to a triangulated equivalence between $D(\cH_\cT)$ and $D(\cG)$ (see~\cite[Theorem~5.7]{FMT} 
and the more general result~\cite[Corollary~2.3]{FMS} which holds also without the filterability property). In this paper we show that these results are true for non-classical tilting modules of finite projective dimension. Let $R$ be a ring and $T \in \rMod R$ be a (non-classical) tilting module of finite projective dimension $n$. Let 
$\cT=(\cT^{\leq0}, \cT^{\geq0})$ be the $t$-structure on $D(R)$ generated by $T$ and $\cD=(\cD^{\leq0}, \cD^{\geq0})$ be the natural $t$-structure. We first give an explicit characterization of the aisle and the co-aisle of $\cT$. Then we prove that the pair $(\cD,\cT)$ satisfies $\cD^{\leq {-n}} \subseteq \cT^{\leq {0}} \subseteq \cD^{\leq {0}}$. Then we show that this pair is right filterable 
As a corollary, we 
obtain~\cite[Corollary~2.3]{FMS} in the case of modules: the heart of $\cT$ is derived equivalent to $\rMod R$, and this equivalence extends the natural inclusion $\cH_\cT \hookrightarrow D(R)$.

\section{Notation and preliminaries}

Given an additive category $\cS$, for any full subcategory $\mathcal{S}'$ of $\mathcal{S}$ we denote by 
$^\perp\mathcal{S}'$ the following full subcategory of $\mathcal{S}$: 
\[
^\perp\mathcal{S}':=\{X\in \mathcal{S} \,|\, \Hom_\mathcal{S}(X,Y)=0, \text{~for all~} Y\in\mathcal{S}'\},
\]
 and by $\mathcal{S}'^\perp$ the following full subcategory of $\mathcal{S}$: 
\[
\mathcal{S}'^\perp:=\{X\in \mathcal{S} \,|\, \Hom_\mathcal{S}(Y,X)=0, \text{~for all~} Y\in\mathcal{S}'\}.
\]

Let $\mathcal{C}$ be a triangulated category. A full subcategory $\cS$ of $\cC$ is called {\it suspended} (resp. {\it cosuspended}) if it is closed under extensions and positive shifts (resp. extensions and negative shifts); $\cS$ is called {\it localizing} if it is suspended and cosuspended, closed under direct summands and under all set-indexed direct sums that exist in $\cC$.
Recall from~\cite{MR751966} that a {\it $t$-structure} in $\cC$ is a pair $\cD=(\cD^{\leq0},\cD^{\geq0})$ of full subcategories of $\cC$ such that, setting $\cD^{\leq n}:=\cD^{\leq0}[-n]$ and $\cDgg n:=\cDgg0[-n]$, one has:
\begin{enumerate}
\item[\rm (i)] $\cDll0\subseteq\cDll1$ and $\cDgg0\supseteq\cDgg1$;
\item[\rm (ii)] $\Hom_\cC(X,Y)=0$, for every $X$ in $\cDll0$ and every $Y$ in $\cDgg1$;
\item[\rm (iii)] For any object $X\in\cC$ there exists a distinguished triangle 
\[A\to X\to B \to A[1]
\]
in $\cC$, with
$A\in\cDll0$ 
and $B\in\cDgg1$.
\end{enumerate}
The classes $\cD^{\leq 0}$ and $\cD^{\geq 0}$ are called the \emph{aisle} and the \emph{co-aisle} of the $t$-structure $\cD$. In such case one has $\mathcal{D}^{\geq0}=(\mathcal{D}^{\leq0})^\perp[1]$ and $\mathcal{D}^{\leq0}=\;^\perp(\mathcal{D}^{\geq1})=\;^\perp((\mathcal{D}^{\leq0})^\perp)$. The objects $A$ and $B$ in the above triangle are uniquely determined by 
$X$, up to isomorphism, and define functors $\delta^{\leq0}\colon \mathcal{D} \to \mathcal{D}^{\leq0}$ and 
$\delta^{\geq0} \colon \mathcal{D} \to \mathcal{D}^{\geq0}$ (called {\it truncation functors}) which are right and left adjoints to the respective inclusion functors. The full subcategory $\mathcal{H}_\cD = \mathcal{D}^{\leq0} \cap \mathcal{D}^{\geq0}$ is called the {\it heart} of the $t$-structure and it is an abelian category, where the short exact sequences ``are'' the triangles in 
$\mathcal{C}$ with their three terms in $\mathcal{H}_\cD$. The assignments 
$X\leadsto(\delta^{\leq0} \circ \delta^{\geq0})(X)$ and 
$X \leadsto (\delta^{\geq0} \circ \delta^{\leq0})(X)$ define naturally isomorphic functors $\mathcal{C} \to \mathcal{H}_\cD$ which are cohomological (see~\cite{MR751966}). 
Given an abelian category $\mathcal{A}$, its derived category $\mathcal{D}(\mathcal{A})$ is a triangulated category which admits a 
$t$-structure $\cD=(\mathcal{D}^{\leq0},\mathcal{D}^{\geq0})$, called the {\it natural $t$-structure}, where 
$\mathcal{D}^{\leq0}$ (resp. $\mathcal{D}^{\geq0}$) is the full subcategory of complexes without cohomology in positive (resp. negative) degrees. The heart of $\cD$ is (naturally equivalent to) $\mathcal{A}$.

Following~\cite{FMT}, we say that a pair of $t$-structures $(\cD,\cT)$ has \emph{shift $k\in\mathbb Z$} and \emph{gap} $n\in\Bbb N$
if $k$ is the maximal number such that $\cT^{\leq k} \subseteq \cD^{\leq 0}$ (or equivalently $\cD^{\geq 0}\subseteq \cT^{\geq k}$) and $n$ is the minimal number such that $\cD^{\leq -n}\subseteq \cT^{\leq k}$ (or equivalently $\cT^{\geq k} \subseteq \cD^{\geq -n}$). Such a $t$-structure will be called of \emph{type} $(n,k)$. Let $\cD$ be a fixed $t$-structure in $\cC$. We say that a $t$-structure $\cT$ in $\cC$ is \emph{right $\cD$-filterable} if for any $i\in\mathbb Z$ the intersection $\cD^{\geq i}\cap \cT^{\geq 0}$ is a co-aisle.
The right filterability is a symmetric notion: 
therefore one says that the pair $(\cD,\cT)$ is \emph{right filterable} if either
$\cT$ is right $\cD$-filterable or equivalently $\cD$ is right $\cT$-filterable. Similarly one defines the notion of \emph{left filterable} pair of $t$-structures (see~\cite{FMT} for more details). 

A pair $(\cD,\cT)$ of $t$-structures in a triangulated category $\cC$ is \emph{$n$-tilting} (resp. \emph{$n$-cotilting}) if:
\begin{enumerate}
\item $(\cD,\cT)$ is filterable of type $(n,0)$; and
\item the full subcategory $\cH_\cD \cap \cH_\cT$ of $\cH_\cD$ cogenerates $\cH_\cD$
(resp. the full subcategory $\cH_\cD \cap \cH_\cT[-n]$ of $\cH_\cD$ generates $\cH_\cD$).
\end{enumerate}
The pair $(\cD,\cT)$ is \emph{right $n$-tilting} (resp. \emph{right $n$-cotilting}) if it is $n$-tilting (resp. $n$-cotilting) and right filterable; \emph{left $n$-tilting} and \emph{left $n$-cotilting} pair of $t$-structures are similarly defined.

\begin{theorem}\cite[Theorem 5.7]{FMT}\label{genheartderiv}
Let $\cA$ be an abelian category and $\cD$ be the natural $t$-structure in $D(\cA)$.
Suppose that $(\cD,\cT)$ is a 
$n$-tilting pair
(resp. 
$n$-cotilting pair)
of $t$-structures; then there is a triangle equivalence
\[
\xymatrix{D(\cH_\cT) \ar[r]^-\simeq& D(\cH_\cD)=D(\cA)}
\]
which extends the natural inclusion $\cH_\cT \subseteq D(\cA)$.
\end{theorem}


\section{The filterability of $t$-structures generated by non-classical tilting modules}

Throughout this section, $R$ is an associative ring with identity and $n$ is a non-negative integer. We recall that a right $R$-module $T$ is called {\it $n$-tilting} if it satisfies the following conditions:
\begin{enumerate}
\item[\rm (T1)]There exists an exact sequence in $\rMod R$
\[
\xymatrix{0\ar[r]& P_n \ar[r]& \dots \ar[r]& P_1\ar[r]& P_0\ar[r]& T\ar[r]&0}
\]
where $P_i$ is projective, for every $i=0,\dots, n$.
\item[\rm (T2)]$\Ext^i_R(T,T^{(\alpha)})=0$, for every $i>0$ and for every cardinal $\alpha$.
\item[\rm (T3)]There exists an exact sequence in $\rMod R$
\[
\xymatrix{0\ar[r]& R \ar[r]& T_0 \ar[r]& T_1\ar[r]& \dots\ar[r]& T_n\ar[r]&0}
\]
where $T_i$ is a direct summand of a direct sum of copies of $T$, for every $i=0,\dots, n$.
\end{enumerate}

\begin{remark}\label{generator}
Let $T \in \rMod R$ be an $n$-tilting module. By axiom (T3), $T$ generates $R$ and so all of the derived category $D(R)$, that is, the smallest localizing subcategory that contains $T$ is $D(R)$. This implies also that for any $X \in D(R)$, if 
$\Hom_{D(R)}(T[i],X)=0$ for every $i\in \bZ$ then $X=0$. Indeed, the full subcategory $\cU$ of $D(R)$ consisting of those objects $Y$ for which $\Hom_{D(R)}(Y[i], X)=0$ is localizing. Since $T \in \cU$ then $D(R) \subseteq \cU$, thus $X=0$ as the identity morphism on $X$ is zero.
\end{remark}

\begin{lemma}\cite[Lemma 3.1, Proposition 3.2]{AJS}\label{ajs}
Let $E$ be a complex of $R$-modules. Let $\Susp_{D(R)}(E)$ be the smallest suspended subcategory of $D(R)$ which is closed under taking set-indexed direct sums and contains $E$. Then $\Susp_{D(R)}(E)$ is an aisle in $D(R)$ and the corresponding co-aisle is
\[
\Susp_{D(R)}(E)^\bot[1]=\left\{B \in D(R) \,|\, \Hom_{D(R)}(E,B[j])=0, \forall j<0  \right\}.
\]
\end{lemma}

Let $T$ be an $n$-tilting module in $\rMod R$. By Lemma~\ref{ajs}, $T$ generates a $t$-structure $\cT=(\cT^{\leq0},\cT^{\geq0})$ in $D(R)$, where:
\[
\cT^{\leq0}=\Susp_{D(R)}(T), \qquad \cT^{\geq0}=\left\{B \in D(R) \,|\, \Hom_{D(R)}(E,B[j])=0, \forall j<0  \right\}.
\]
We have:

\begin{lemma}\label{aisle}
Let $T$ be a $n$-tilting module. Then
\[
\Susp_{D(R)}(T)=\left\{A \in D(R) \,|\, \Hom_{D(R)}(T,A[j])=0, \forall j>0  \right\}.
\]
\end{lemma}
\begin{proof}
Let $\cU$ be the following full subcategory of $D(R)$:
\[
\cU=\left\{A \in D(R) \,|\, \Hom_{D(R)}(T,A[j])=0, \forall j>0  \right\}.
\]
We shall show that $\cU$ is a suspended subcategory of $D(R)$ containing $T$, that $\cU$ is closed under taking direct sums in $D(R)$ and that $\cU\subseteq \Susp_{D(R)}(T)$.

First of all, $T$ belongs to $\cU$ by axiom (T2). It is clear that $\cU$ is closed under positive shifts. Let us prove that $\cU$ is closed under extensions. Suppose that 
\[
A\to B\to C \stackrel{+}\to
\]
is a distinguished triangle in $D(R)$ with $A, C\in \cU$. Let $j>0$ and $f \colon T \to B[j]$ be a morphism in $D(R)$. As $C\in \cU$ we get that $\Hom_{D(R)}(T,C[j])=0$, so $f$ factors through a morphism $T \to A[j]$, which is the zero morphism since $A\in \cU$. Hence $f=0$. This shows that $B\in \cU$.

Now let $(X_i)_{i\in I}$ be a family of objects of $\cU$ and let us prove that their direct sum in $D(R)$ also belong to $\cU$. Both the direct sum $\oplus_{i\in I}X_i$ and the product $\prod_{i\in I}X_i$ in $D(R)$ are given by taking termwise direct sums and products, respectively, of the family of complexes $(X_i)_{i\in I}$. Hence there is a canonical monomorphism of complexes $\oplus_{i\in I}X_i \hookrightarrow \prod_{i\in I}X_i$. By axiom (T1), $T$ has a projective resolution $P^\bullet \to T$ of length $n$. 
Let $j>0$ and consider the following diagram:
\[
\xymatrixcolsep{0.1pc}\xymatrix{\Hom_{D(R)}(T, \oplus_{i\in I}X_i [j]) && \Hom_{D(R)}(T, \prod_{i\in I}X_i[j]) \simeq \prod_{i\in I}\Hom_{D(R)}(T, X_i[j])=0\\
\Hom_{K(R)}(P^\bullet, \oplus_{i\in I}X_i [j])\ar[rr]\ar[u]^{\simeq} && \Hom_{K(R)}(P^\bullet, \prod_{i\in I}X_i [j])\ar[u]^{\simeq}\\
\Hom_{Ch(R)}(P^\bullet, \oplus_{i\in I}X_i [j])\ar@{^{(}->}[rr]\ar@{->>}[u] && \Hom_{Ch(R)}(P^\bullet, \prod_{i\in I}X_i [j])\ar@{->>}[u]\\
}
\]
This shows that any morphism in $\Hom_{Ch(R)}(P^\bullet, \prod_{i\in I}X_i [j])$ is nullhomotopic, hence so is any morphism in $\Hom_{Ch(R)}(P^\bullet, \oplus_{i\in I}X_i [j])$. It follows that \[
\Hom_{D(R)}(T, \oplus_{i\in I}X_i [j])=0,
\]
so $\oplus_{i\in I}X_i \in \cU$. Notice that by what we have showed above, from the minimality of $\Susp_{D(R)}(T)$ we get that 
$\Susp_{D(R)}(T) \subseteq \cU$.

Finally, let us show that $\cU\subseteq \Susp_{D(R)}(T)$. Let $X \in \cU$ and let 
\[
A_X\to X\to B_X \stackrel{+}\to
\]
be the distinguished triangle of $X$ with respect to the $t$-structure $\cT$. We prove that $B_X=0$. To do this, we use the fact that 
$T$ is a generator of $D(R)$ (see Remark~\ref{generator}). So let $i \in \bZ$ and $f \colon T[i] \to B$ be a morphism in $D(R)$. If $i\geq 0$ then $f=0$ because $B \in \Susp_{D(R)}(T)^\bot$. If $i<0$ then $\Hom_{D(R)}(T[i], A_X[1])=0$ since $1-i>0$ and $A_X \in \Susp_{D(R)}(T) \subseteq \cU$. Then $f$ factors through a morphism $T[i] \to X$, which is zero because $X \in \cU$, hence $f=0$. This shows that $\Hom_{D(R)}(T[i], B_X)=0$ for every $i\in \bZ$. Therefore $B_X=0$ and $X\simeq A_X \in \Susp_{D(R)}(T)$.
\end{proof}

Let $\cD=(\cD^{\leq0},\cD^{\geq0})$ be the natural $t$-structure on $D(R)$. From the natural isomorphism $\Hom_{D(R)}(R,A[j])\simeq H^j(A)$, we see that the $t$-structure generated by the $0$-tilting module $R$ coincides with $\cD$. Hence the aisle $\cD^{\leq0}$ (respectively, the co-aisle $\cD^{\geq0}$) consists of those objects $A$ of $D(R)$ for which $\Hom_{D(R)}(R,A[j])=0$, for every $j>0$ (respectively, for every $j<0$).

\begin{lemma}\label{type}
The pair $(\cD,\cT)$ is of type $(n,0)$.
\end{lemma}
\begin{proof}
We shall show that $\cD^{\leq{-n}}\subseteq \cT^{\leq0}\subseteq \cD^{\leq0}$, with $n$ minimal and $0$ maximal integers realizing this chain of inclusions. Since $T \in \cD^{\leq 0}$, for every $Y\in \cD^{\geq0}$ and every $j<0$ one has that 
\[
\Hom_{D(R)}(T,Y[j])=0.
\]
This shows that $\cD^{\geq 0}\subseteq \cT^{\geq0}$ and so, by reversing the inclusions, $\cT^{\leq0}\subseteq \cD^{\leq0}$. Moreover, as $T$ does not lie in $\cT^{\geq1}$ we see that $\cD^{\geq 0}\not\subseteq \cT^{\geq1}$, hence $\cT^{\leq 1} \not\subseteq \cD^{\leq 0}$. It remains to prove that $\cD^{\leq{-n}}\subseteq \cT^{\leq0}$. So let us pick an object $X\in \cD^{\leq{-n}}$. First assume that $X$ is bounded below and let 
$k={\rm min}\left\{m\in \mathbb{Z} \,|\, H^m(X)\neq0  \right\}$. Clearly, $k\leq -n$. Let $j>0$ and $f \colon T\to X[j]$ be a morphism in $D(R)$. Since $T$ has projective dimension $n$, we have that $\Hom_{D(R)}(T, H^{-n-j}(X[j])[n+j])=\Ext_R^{n+j}(T, H^{-n-j}(X[j]))=0$, so $f$ factors trought a morphism $f_1\colon T \to \delta^{\leq{-n-j-1}}(X[j])$. By iterating this procedure we may write $f$ as the composition $f=g\circ f_{-n-k}$, for a suitable morphism $g$ and where $f_{-n-k} \colon T \to H^{k-j}(X[j])[-k+j]$. The latter is the zero morphism since $-k+j\geq n+j>n$. Therefore $f=0$. Now let us assume that $X$ is not bounded below. We can write $X$ as the homotopy colimit of its brutal truncations. More precisely, let us define for every $i\geq0$ the complex $\sigma_{\geq -n-i}(X)$ as follows:
\[
\sigma_{\geq -n-i}(X)^p =
\left\{
\begin{array}{rl}
X^p, & \mbox{if } p \leq -n-i, \\
0, & \mbox{if } p>-n-i
\end{array}
\right.
\] 
Then $X=\hocolim_{i\geq0}\sigma_{\geq -n-i}(X)$. Now for every $i\geq0$ the complex $\sigma_{\geq -n-i}(X)$ is bounded below and belongs to $\cD^{\leq{-n}}$, so by what we have proved above we have $\sigma_{\geq -n-i}(X) \in \cT^{\leq0}$. But $\cT^{\leq0}$ is closed under taking homotopy colimits in $D(R)$ (since it is suspended and closed under coproducts), so $X\in \cT^{\leq0}$. This proves that $\cD^{\leq{-n}}\subseteq \cT^{\leq0}$. Finally, since $T$ is $n$-tilting there exists an $R$-module $M$ such that 
$\Hom_{D(R)}(T,M[n])\neq 0$. Then $M[n-1] \in \cD^{\leq{-n+1}}$, but $M[n-1]\not\in\cT^{\leq0}$. This shows that $\cD^{\leq{-n+1}} \not\subseteq \cT^{\leq0}$.
\end{proof}

\begin{proposition}\label{filt}
The pair $(\cD,\cT)$ is right filterable, that is for any $i\in \mathbb{Z}$ the intersection $\cD^{\geq i}\cap\cT^{\geq 0}$ is a co-aisle.
\end{proposition}
\begin{proof}
Let us fix an integer $i$. Using Lemma~\ref{ajs} we see that the full subcategory $\Susp_{D(R)}(R[-i]\oplus T)$ is the aisle of a $t$-structure, whose corresponding co-aisle is 
\begin{flalign*}
&\Susp_{D(R)}(R[-i]\oplus T)^\bot[1] = \left\{B \in D(R) \,|\, \Hom_{D(R)}(R[-i]\oplus T,B[j])=0, \forall j<0  \right\}&& \\ 
              &=  \scriptstyle{\left\{B \in D(R) \,|\, \Hom_{D(R)}(R[-i],B[j])=0, \forall j<0  \right\}} \cap \scriptstyle{\left\{B \in D(R) \,|\, \Hom_{D(R)}(T,B[j])=0, \forall j<0  \right\}} &&\\
              &=  \cD^{\geq {i}}\cap\cT^{\geq 0}.
\end{flalign*}
This gives the desired result.
\end{proof}

\begin{theorem}\label{ntilting}
The pair $(\cD,\cT)$ is $n$-tilting.
\end{theorem}
\begin{proof}
By Lemma~\ref{type} and Proposition~\ref{filt} the pair $(\cD,\cT)$ is right filterable of type $(n,0)$. The full subcategory $\rMod R \cap \cH_\cT$ of $\rMod R$ is cogenerating since it contains all the injective modules. Hence the result follows.
\end{proof}

We are now able to recover~\cite[Corollary 2.3]{FMS} in the case of modules over a ring:

\begin{corollary}\label{dereq}
Let $R$ be a ring, let $T\in \rMod R$ be an $n$-tilting module and let $\cH_T$ be
the heart of the associated t-structure $\cT=(\cT^{\leq0},\cT^{\geq0})$ in $D(R)$. The inclusion
functor $\cH_\cT \hookrightarrow D(R)$ extends to a triangulated equivalence $D(\cH_\cT) \stackrel{\sim}\to D(A)$.
\end{corollary}
\begin{proof}
This follows from Theorem~\ref{genheartderiv} and Theorem~\ref{ntilting}. 
\end{proof}


\end{document}